\documentclass[a4paper]{amsart}
\usepackage{mathrsfs,amssymb,amsmath,graphicx,amsthm}
\usepackage[pdfauthor={Morten S. Risager and Jimi L. Truelsen},
pdftitle={Distribution of angles in hyperbolic lattices}
pdfsubjects={Primary 11J71; Secondary },
pdfkeywords={},
            pdfproducer={LaTeX2e with hyperref},
            pdfcreator={Pdflatex},
            pdfdisplaydoctitle=true
]{hyperref}

\usepackage[labelsep=none]{caption}

\usepackage{ifpdf} 
\ifpdf 
\fi

\newtheorem{theorem}{Theorem }
\newtheorem{lemma}{Lemma}

\theoremstyle{definition}

\theoremstyle{remark}

\def \p {{\varphi}}
\def \fortegn {\hbox{sign}}
\def \angle {{\varphi}}
\def \a {{\alpha}}
\def \G {{\Gamma}}
\def \g {{\gamma}}
\def \R {{\mathbb R}}
\def \H {{\mathbb H}}

\def \Z {{\mathbb Z}}

\def \GmodH {{\Gamma\setminus\H}}
\def \vol {\hbox{vol}}

\def \slr  {\hbox{SL}_2(\mathbb R)}

\newcommand{\inprod}[2]{\left \langle #1,#2 \right\rangle}
\newcommand{\mattwo}[4]
{\left(\begin{array}{cc}
                        #1  & #2   \\
                        #3 &  #4
                          \end{array}\right) }

\newcommand{\abs}[1]{\left\lvert #1 \right\rvert}
\newcommand{\norm}[1]{\left\lVert #1 \right\rVert}
\begin{document}
\date{\today}

\thanks{The first author was funded by a Steno Research Grant from The
  Danish Natural Science Research Council, and the second author was
  supported by a stipend (EliteForsk) from The Danish Agency for Science, Technology and Innovation}
\title{Distribution of Angles in Hyperbolic Lattices}    
\author{Morten S. Risager}
\address{Department of Mathematical Sciences, University of
  Copenhagen, Universitetsparken 5, 2100 Copenhagen \O, Denmark}
\email{risager@math.ku.dk}
\author{Jimi L. Truelsen}
\address{Department of Mathematical Sciences, University of Aarhus, Ny Munkegade Building 1530, 8000 Aarhus C, Denmark}
\email{lee@imf.au.dk}


\begin{abstract}
We prove an effective equidistribution result about angles in a
hyperbolic lattice. We use this to generalize a result due to F. P. Boca.
\end{abstract}

\maketitle
\section{Introduction}
Consider the group $G =
\slr$ that acts on the upper halfplane $\H$ by linear fractional
transformations. Let $\G\subset G$ be a cofinite discrete
group, and let $d:\H\times\H\to \R_+$ denote the hyperbolic distance. Consider the counting function
\begin{equation*}
  N_\G(R,z_0,z_1)=\#\{\g\in \G\lvert\, d(z_0,\g z_1)\leq R\}.
\end{equation*}
The \emph{hyperbolic lattice point problem} is the problem of estimating this function as $R\to\infty$. A typical result would be an asymptotic expansion of the form 
\begin{equation}\label{hyperbolic lattice}
  N_\G(R,z_0,z_1)=\frac{\kappa_\G\pi}{\vol(\GmodH)}e^R+O(e^{ R(\a+\varepsilon)})
\end{equation}
for some $\a<1$, where $\kappa_\G = 2$ if $-I\in \G$ and $\kappa_\G = 1$ otherwise.
 The problem has been considered by numerous people including Delsarte \cite{Delsarte:1942a}, Huber \cite{Huber:1959a,Huber:1960a,Huber:1961a} ($\G$ cocompact),
 Patterson \cite{Patterson:1975a} ($\a=3/4$ if there are no small eigenvalues), Selberg (unpublished) and
 Good \cite{Good:1983b} ($\a=2/3$ if there are no small eigenvalues). Higher dimensional analogues have
 also been considered (see e.g \cite{LaxPhillips:1982a,Levitan:1987a,
 ElstrodtGrunewaldMennicke:1988a}), as well as the analogous problem
 for manifolds with non-constant curvature \cite{Margulis:1969a,
 Gunther:1980a}. For a discussion of the optimal choice of $\a$ we
 refer to \cite{PhillipsRudnick:1994a}, where the authors prove that 
 $\a$ must be at least $1/2$ and they indicate that in many
 cases we should maybe expect (\ref{hyperbolic lattice}) to hold with
 $\a=1/2$. 

Let $\angle_{z_0,z_1}(\g)$ be $(2\pi)^{-1}$ times the angle between the vertical geodesic from $z_0$ to $\infty$ and the geodesic between $z_0$ and $\g z_1$.
\begin{figure}[h]
\begin{center}
  \includegraphics{fig1.1}
\end{center}
\caption{}  
\end{figure}

These normalized angles are equidistributed modulo one, i.e. for every interval $I\subset \R\slash\Z$ we have 
\begin{equation}\label{equidistribution-angles}
  \frac{N_\Gamma^I(R,z_0,z_1)}{N_\Gamma(R,z_0,z_1)}\to \abs{I} \textrm{ as } R\to \infty,
\end{equation}
where
\begin{equation}
  N_\Gamma^I(R,z_0,z_1)=\#\{\g\in \G\lvert\, d( z_0, \g z_1)\leq R, \, \angle_{z_0,z_1}(\g)\in I\},
\end{equation}
and $\abs{I}$ is the length of the interval. This has been proved by Selberg (unpublished, see comment in \cite[p. 120]{Good:1983b}), Nicholls \cite{Nicholls:1983a} and Good  \cite{Good:1983b}.

In this paper we start by proving (\ref{equidistribution-angles}) with an error term:
\begin{theorem}\label{effective equidistribution} Let $K\subset \H$
  be a compact set. There exists a constant $\a<1$ possibly
  depending on $\G$ and $K$ such that for all $z_0,z_1\in K$ and all intervals $I$ in $\R\slash\Z$
\begin{equation*}
   \frac{N_\Gamma^I(R,z_0,z_1)}{N_\Gamma(R,z_0,z_1)}=\abs{I}+O(e^{R(\a-1+\varepsilon)}).
\end{equation*}
\end{theorem}
If we assume that the automorphic Laplacian on $\GmodH$ has no
exceptional eigenvalues, i.e. eigenvalues in $]0,1/4[$,  we prove that we can take
\begin{equation*}
  \a=11/12.
\end{equation*}
If there are exceptional eigenvalues the exponent could become larger, depending on how close to zero they are. We prove Theorem \ref{effective equidistribution} by proving asymptotic expansions for the exponential sums
\begin{equation}\label{exp-sum}
  \sum_{\substack{\g\in \G\\ d( z_0,\g z_1)\leq R}}e(n\angle_{z_0,z_1}(\g)),
\end{equation}
where  $n\in \Z$ and $e(x)=\exp(2\pi i x)$. The exponent $11/12$ can certainly be improved. In fact our proof uses
a variant of Huber's method \cite{Huber:1959a} which does not give the
optimal bound even for the expansion (\ref{hyperbolic lattice}). In
principle Theorem \ref{effective equidistribution} could be proved by
using the method of Good from \cite{Good:1983b}, which gives the best
known error term in the hyperbolic lattice point problem (\ref{hyperbolic lattice}). The one missing point in \cite{Good:1983b} to prove Theorem \ref{effective equidistribution} is the dependence of $n$ in the expansion of the exponential sum (\ref{exp-sum}). Rather than patiently tracking down the $n$-dependence, we found it more to the point -- albeit at the expense of poor error terms -- to provide an alternative and more direct proof inspired by \cite{Huber:1959a}.


Recently Boca \cite{Boca:2007a} considered a related problem: What
happens if we order the elements according to $d( z_1,\g z_1)$ instead
of $d(z_0,\g z_1)$? Let $\G(N)$ be the principal congruence group of level $N$ i.e. the set of $2\times2$ matrices $\g$ satisfying $\g\equiv I \mod N$. Boca identified for these groups the limiting distribution using
non-trivial bounds for Kloosterman sums. He proved the following\footnote{ Readers
  consulting \cite{Boca:2007a} should be warned that our notation differs slightly from Boca's.}:
Let $z_0, z_1 \in \H$ and 
  let $\omega_{z_0,z_1}(\g)$ denote the angle
  in $[-\pi/2,\pi/2]$ between the vertical geodesic through $z_0$ and
  the the geodesic containing $z_0$ and $\g z_1$ (if $z_0 = \g z_1$ you
  can assign $\omega_{z_0,z_1}(\g)$ the value $0$ -- it does not matter
  what you choose, since there are only a finite number of such
  $\gamma$'s). 
\begin{figure}[h]
\begin{center}
  \includegraphics{fig2.1}
\end{center}
\caption{}
\end{figure}

For any interval $I \subset [-\pi/2,\pi/2]$ we consider the counting function 
\begin{equation*}
   \mathfrak{N}_\G^I(R,z_0,z_1) =\#\{\g\in \G\mid\, d( z_1,\g z_1)\leq R, \,
    \omega_{z_0,z_1}(\g)\in I\}.
\end{equation*}
We emphasize that the elements are ordered according to $d(z_1,\g z_1)$ instead of  $d(z_0,\g z_1)$. We shall write $\mathfrak{N}_\G(R,z_0,z_1)$ instead of $\mathfrak{N}_\G^{[-\pi/2,\pi/2]}(R,z_0,z_1)$. Following Boca we define 
\begin{align*}
\eta_{z_0,z_1} (t)=
\frac{2y_0y_1(y_0^2+y_1^2+(x_0-x_1)^2)}{(y_0^2+y_1^2+(x_0-x_1)^2)^2 -
  ((y_1^2-y_0^2+(x_0-x_1)^2)\cos(t)+2y_0(x_0-x_1)\sin(t))^2}.
\end{align*}
Then Boca proves the following result:
\begin{theorem}\label{boca-result} Let $\G=\G(N)$.  For any interval $I \subset [-\pi/2,\pi/2]$ 
  \begin{equation*}
 \frac{\mathfrak{N}_{\G}^I(R,z_0,z_1)}{\mathfrak{N}_{\G}(R,z_0,z_1)}     =
 \frac{1}{\pi}\int_I \eta_{z_0,z_1}(t)dt+O(e^{(7/8-1+\varepsilon)R})
  \end{equation*}
for any $\varepsilon > 0$.
\end{theorem}
In the view of (\ref{hyperbolic lattice}) Theorem \ref{boca-result} is equivalent to an expansion of $\mathfrak{N}_{\G(N)}^I(R,z_0,z_1)$.
We generalize and refine Boca's result: With data as above, $I \subset
\R/\Z$ and $w\in \H$ we consider the counting function 
\begin{equation*}
  \mathscr{N}_\G^I(R,z_0,z_1,w) =\#\{\g\in \G\lvert\, d( z_1, \g w)\leq R, \, \angle_{z_0,w}(\g)\in I\}.
\end{equation*}
We emphasize that we order according to $d( z_1, \g w)$. As before we shall write $\mathscr{N}_\G(R,z_0,z_1,w)$ instead of $\mathscr{N}_\G^{[-1/2,1/2]}(R,z_0,z_1,w)$. Besides the more general ordering our result is more
  refined in the sense that we can distinguish between angles
  that differ by $\pi$. Consider 
\begin{equation*}
  \rho_{z_0,z_1}(\omega)= \frac{2y_0 y_1}{((x_0-x_1)^2+y_0^2+y_1^2)(1-\cos(2\pi\omega))+2y_0^2\cos(2\pi\omega)+2(x_1-x_0)y_0\sin(2\pi\omega)}.
\end{equation*}
Then we prove the following result:
\begin{theorem}\label{boca-generalized} Let $\G$ be any
  cofinite Fuchsian group. There exists $\alpha<1$ such that
  for any $I \subset \R/\Z$ we have 
  \begin{equation*}
   \frac{\mathscr{N}_\G^I(R,z_0,z_1,w)}{\mathscr{N}_\G(R,z_0,z_1,w)}
   =\int_I \rho_{z_0,z_1}(\omega)d\omega+O(e^{(\a-1+\varepsilon)R})
  \end{equation*}
for any $\varepsilon > 0$.
\end{theorem}
Note that in the special case of $\G=\G(N)$ and $w=z_1$ this implies Theorem \ref{boca-result} (with a poorer error term though), since
\begin{equation*}
\eta_{z_0,z_1}(2\pi t) = \rho_{z_0,z_1}(t) + \rho_{z_0,z_1}(t+1/2).
\end{equation*}
We will prove that Theorem \ref{boca-generalized} follows from Theorem \ref{effective equidistribution}. 

Whereas Boca is using a non-trivial bound for Kloosterman sums, we are
utilizing the fact that for \emph{any} group there is a spectral gap between the zero eigenvalue of the Laplacian and the first non-zero eigenvalue. As in Theorem \ref{effective equidistribution} the $\a$ in Theorem \ref{boca-generalized} generally depends on the size of the first non-zero eigenvalue.

We remark that all the results presented here, can easily be
 phrased in terms of points in the orbit $\G z_1$, rather than elements
 in $\G$, since
 \begin{align*} \#\{z \in \G z_1 \mid d(z_0,z)\leq R\} =
   \frac{N_\G(R,z_0,z_1)}{\vert \G_{z_1} \vert}, 
\end{align*}
where $\G_{z_1}$ denotes the stabilizer of $z_1$.\\

\section{Effective equidistribution of angles}
Let $G=\slr$. The group $G$ acts on the upper halfplane $\H$ by linear fractional transformations 
\begin{equation*}
  gz=\frac{az+b}{cz+d}, \ g=\mattwo{a}{b}{c}{d}\in G,\, z\in \H.
\end{equation*}
Let $\G\subset \slr$ be discrete and cofinite. For simplicity we assume that $-I\notin \G$. If $-I\in \G$ we need to multiply all main terms by $2$.

For $z\in \H$ we let $r=r(z)$ and $\p=\p(z)$ be the geodesic polar coordinates of $z$. These are related to the rectangular coordinates by
\begin{equation}
  \label{eq:1}
  z=\mattwo{\cos \p(z)}{\sin\p(z)}{-\sin\p(z)}{\cos{\p(z)}}\exp{(-r(z))}i.
\end{equation}
We note that if $z_0=x_0+iy_0$ and we let
\begin{equation*}
  \g_0=\mattwo{1/\sqrt{y_0}}{-x_0/\sqrt{y_0}}{0}{\sqrt{y_0}}
\end{equation*}
then it is straightforward to check that $\g_0 z_0=i$. We see that
\begin{align*}
  \angle_{z_0,z_1}(\g)&=\angle_{i,\g_0 z_1}(\g_0\g\g_0^{-1})=\p(\g_0\g\g_0^{-1}(\g_0 z_1) )/\pi
\end{align*}
and
\begin{align*}
d(z_0,\g z_1)&=d(i,\g_0\g\g_0^{-1}(\g_0 z_1) )=r(\g_0\g\g_0^{-1}(\g_0 z_1) ).
\end{align*}
Therefore after conjugation of the group $\G$ the counting problems in
the introduction may be formulated in terms of $r(\g z)$ and $\p(\g
z)$ with $z=\g_0 z_1$.

The Laplacian for the $G$-invariant measure $d\mu(z)=dxdy/y^2$
on $\H$ is given in Cartesian coordinates by 
\begin{equation*}
  \label{eq:13}
\Delta=y^2\left(\frac{\partial^2}{\partial x^2}+\frac{\partial^2}{\partial y^2}\right).
\end{equation*}

In geodesic polar coordinates the Laplace operator is given by
\begin{equation}
  \label{eq:2}
  \Delta=\frac{\partial^2 }{\partial r^2}+\frac{1}{\tanh r}\frac{\partial }{\partial r}+\frac{1}{4\sinh^2 (r)}\frac{\partial^2 }{\partial \p^2}.
\end{equation}

Consider  $L^2(\Gamma\backslash\H,d\mu(z))$ with
inner product $\inprod{f}{g}=\int_{\Gamma\backslash
  \H}f\overline{g}d\mu(z)$ and norm
$\norm{f}_2=\sqrt{\inprod{f}{f}}$. The Laplacian induces an operator
on $L^2(\Gamma\backslash\H,d\mu(z))$ called the automorphic Laplacian defined as follows:
Consider the operator defined by $-\Delta f$ on  smooth,  bounded,
$\Gamma$-invariant functions satisfying that $-\Delta f$ is also
bounded. This operator is densely defined in
$L^2(\Gamma\backslash \H)$ and is in fact essentially selfadjoint. The closure of this operator is called the automorphic
Laplacian. By standard abuse of notation we also denote this operator by $-\Delta$ 

The automorphic Laplacian is selfadjoint and non-negative and has
eigenvalues 
\begin{equation*}
  0=\lambda_0<\lambda_1\leq \lambda_2 \leq\ldots \lambda_i \leq \ldots 
\end{equation*}
with the number of eigenvalues being finite or $\lambda_i\to\infty$.
It has a continuous spectrum $[1/4, \infty[$ with multiplicity equal to
the number of inequivalent cusps.

By standard operator theory for selfadjoint operators (See
e.g. \cite{Kato:1976a}) the resolvent $R(s)=(-\Delta-s(1-s))^{-1}$ is a
bounded operator which is
meromorphic in $s$ for $s(1-s)$ \emph{off} the spectrum of $-\Delta$. For an eigenvalue $\lambda_i$ outside the continuous spectrum
the operator $R(s)-P_i/(\lambda_i-s(1-s))$ is analytic at $s$ satisfying
$s(1-s)=\lambda_i$ where $P_i$ is the projection to the
$\lambda_i$-eigenspace. In particular for $\lambda=0$ we note that \begin{equation}\label{poleat1}R(s)-\frac{P_0}{-(s(1-s))}\end{equation} is analytic for
$\Re(s)>1-\delta$ for some $\delta$ where $P_0 f=\int f(z)
d\mu{(z)}/\vol{(\Gamma\backslash\H)}$ is the projection to the
$0$-eigenspace. (Alternatively one may quote \cite[Theorem
7.5]{Iwaniec:2002a} to obtain
the same result.)

We define for $\Re(s)>1$
\begin{equation}
  \label{eq:3}
  G_n(z,s)=\sum_{\g\in \G}\frac{e(n\p(\g z)/\pi)}{(\cosh (r(\g z)))^s}.
\end{equation}
 We recall that 
\begin{equation}
  \label{eq:5}
  \cosh (r(\g z))=1+2u(\gamma z,i),
\end{equation}
where $u(z,w)$ is the point pair invariant defined by
\begin{equation}
  \label{eq:6}
  u(z,w)=\frac{\abs{z-w}^2}{4\Im(z)\Im(w)}.
\end{equation}
Hence 
\begin{equation*}
\abs{  \frac{e(n\p( z)/\pi)}{(\cosh (r( z)))^s}}\leq \frac{1}{(1+2u(z,i))^{\Re(s)}}.
\end{equation*}

It therefore follows from \cite[Theorem 6.1]{Selberg:1989a} and the
discussion leading up to it that the sum (\ref{eq:6}) converges
absolutely and uniformly on compact sets and the limit is
$\G$-automorphic,  and bounded in $z$ -- in particular square integrable
on $\GmodH$.

By applying the Laplace operator to $G_n(z,s)$ a straightforward calculation shows that 
\begin{equation}
  \label{eq:4}
  (-\Delta -s(1-s))G_n(z,s)=s(s+1)G_n(z,s+2)+\sum_{\g\in \G}\frac{  n^2e(n\p(\g z)/\pi)}{\sinh^2 (r(\g z))(\cosh (r(\g z)))^s}.
\end{equation}
The sum on the right converges absolutely and uniformly on compacta for $\Re(s)>-1$. Since $G_n(z,s)$ is square integrable, we may invert (\ref{eq:4}) using the resolvent
\begin{equation}
  \label{eq:7}
  R(s)=(-\Delta-s(1-s))^{-1},
\end{equation}
so we have
\begin{equation}
  \label{eq:8}
  G_n(z,s)=R(s)\left(s(s+1)G_n(z,s+2)+\sum_{\g\in \G}\frac{ n^2e(n\p(\g z)/\pi)}{\sinh^2 (r(\g z)) (\cosh (r(\g z)))^s}\right).
\end{equation} 
The right-hand-side is meromorphic in $s$ for $\Re(s)>1/2$ since the
resolvent is holomorphic for $s(1-s)$ not in the spectrum of the
automorphic Laplacian. This gives the meromorphic continuation of
$G_n(z,s)$ to $\Re(s)>1/2$. The only potential poles are at $s=1$ and
$s=s_j$ where $s_j(1-s_j)$ is a small eigenvalue for the automorphic
Laplacian. Using the analyticity of (\ref{poleat1}) we see that the pole at $s=1$ has residue
\begin{equation}
  \label{eq:9}
  \frac{1}{\vol{(\GmodH)}}\int_{\GmodH}\left(2G_n(z,3)+\sum_{\g\in \G}\frac{ n^2e(n\p(\g z)/\pi)}{\sinh^2 (r(\g z)) \cosh (r(\g z))}\right)d\mu(z).
\end{equation}
By unfolding the integral we find that this equals
 \begin{equation}
   \label{eq:9b}
   \frac{1}{\vol{(\GmodH)}}\int_\H\left(2\frac{e(n\p(z)/\pi)}{\cosh^3 (r(z))}+\frac{ n^2e(n\p(z)/\pi)}{\sinh^2 (r(z)) \cosh (r(z))}\right)d\mu(z).
 \end{equation}
Changing to polar coordinates we find
\begin{equation}
  \label{eq:10}
  \frac{1}{\vol{(\GmodH)}}\int_0^\infty\int_0^{\pi}\left(2\frac{e(n\p/\pi)}{\cosh^3 (r)}+\frac{  n^2e(n\p/\pi)}{\sinh^2 (r) \cosh (r)}\right)2\sinh (r) d\p dr,
\end{equation}
which equals
\begin{equation}
  \label{eq:11}
  \frac{2\pi\delta_{n=0}}{\vol{(\GmodH)}}.
\end{equation} 
This follows since
\begin{align*}
\int_0^\infty \frac{2\sinh(r)}{\cosh(r)^3}dr = 1.
\end{align*}

From a Wiener-Ikehara Tauberian theorem (see e.g. \cite[Theorem 3.3.1
and Exercises 3.3.3+3.3.4]{Murty:2001a}) we may conclude that
\begin{equation}
  \label{eq:12}
\sum_{\substack{\g\in \G\\\cosh(r(\g z))\leq R}}\!\! e(n\p(\g)/\pi) =2\pi \frac{\delta_{n=0}}{\vol{(\GmodH)}}R+o(R).
\end{equation}
This implies immediately -- via Weyl's criterion -- that the angles $\p(\g)/\pi$ are equidistributed modulo 1.  

Since we intend to obtain a power saving in the remainder term we investigate $G_n(z,s)$ a bit more careful:

\begin{lemma}\label{growth} Write $s=\sigma+it$. For $z$ in a fixed
  compact set $K \subset \H$, $\abs{t}>1$ and $\sigma>\sigma_0>1/2$ we have 
  
  \begin{equation*}
    G_n(z,s)=O(\abs{t}(\abs{t}^2+n^2)),
  \end{equation*}
where the implied constant may depend on $\G$, $K$, and $\sigma_0$. 
\end{lemma}
\begin{proof}
We recall that  \cite[V (3.16)]{Kato:1976a}
\begin{equation}
  \label{eq:16}
  \norm{R(s)}_\infty\leq \frac{1}{\hbox{dist}(s(1-s),\hbox{spec}(-\Delta))}\leq  \frac{1}{\abs{t}(2\sigma-1)},
\end{equation}
where $\Vert \cdot \Vert_\infty$ denotes the operator norm.
For $\sigma>3/2$ we have 
\begin{equation}
  \label{eq:14}
\norm{G_n(z,s)}_2\leq \norm{G_0(z,3/2)}_2=O(1).
\end{equation}
For $\sigma>\sigma_0$
we may use this and (\ref{eq:8}) to conclude that
\begin{align}\label{eq:for15}
\nonumber  \norm{G_n(z,s)}_2&\leq \norm{R(s)}_\infty\left(\norm{s(s+1)G_n(z,s+2)}_2+\norm{\sum_{\g\in \G}\frac{ n^2e(n\p(\g z)/\pi)}{\sinh^2 (r(\g z)) (\cosh (r(\g z)))^s}}_2\right)\\
&\leq \frac{1}{\abs{t}(2\sigma-1)}\left(\abs{t}^2\norm{G_0(z,3/2)}_2+\norm{\sum_{\g\in \G}\frac{ n^2}{\sinh^2 (r(\g z)) (\cosh (r(\g z)))^{1/2}}}_2\right)\\
&=O(\abs{t}^{-1}(\abs{t}^2+n^2)).\nonumber
\end{align}

Using this and (\ref{eq:4}) we find
\begin{equation}
  \label{eq:15}
  \norm{\Delta G_n(z,s)}_2=O(\abs{t}(\abs{t}^2+n^2)).
\end{equation}
We can now use the Sobolev embedding theorem and elliptic regularity
theory to get a pointwise bound:

For any non-empty open set $\Omega$ in $\R^2$ we consider the
classical Sobolev
space $W^{k,p}(\Omega)$ with corresponding norm
$\norm{\cdot}_{W^{k,p}(\Omega)}$ (See
\cite[p. 59]{AdamsFournier:2003a}). Whenever $\Omega$ satisfies the
cone condition (See \cite[p. 82]{AdamsFournier:2003a}) the Sobolev
embedding theorem \cite[Thm 4.12]{AdamsFournier:2003a}) gives an embedding
\begin{equation}
  \label{eq:20}
  W^{2,2}(\Omega)\to C_B(\Omega)
\end{equation} where $C_B(\Omega)$ is the set of bounded continuous
functions on $\Omega$ equipped with the sup norm. In particular for
$f\in W^{2,2}(\Omega)$ we have 
\begin{equation}
  \label{eq:21}
  \sup_{z\in \Omega}\abs{f(z)}\leq C\norm{f}_{W^{2,2}(\Omega)}
\end{equation}
where $C$ is a constant which depends only on $\Omega$. 

By elliptic regularity theory, if $\Delta_E=\partial^2/\partial
x^2+\partial^2/\partial y^2$ is the Euclidean Laplace operator we have
also that if $u\in W^{1,2}(\Omega)$ satisfies $\Delta_E u\in
L^2(\Omega)$ (weak derivative) then
\begin{equation}\label{eq:22}
\norm{u}_{W^{2,2}(\Omega')}\leq
C'(\norm{u}_{L^2(\Omega)}+\norm{\Delta_E u}_{L^2(\Omega)})
\end{equation}
for all  $\Omega'\subset \Omega$ which satisfies that the closure of
$\Omega'$ is compact and contained in $\Omega$. Here  $C'$ is a constant which depends only on $\Omega$, and $\Omega'$
(See \cite[Theorem 8.2.1]{Jost:2002a}).

We can use this general theory to bound $\abs{G_n(z,s)}$ in the following way: For every $z$ in
the compact set $K$ we fix a small open (Euclidean) disc $\Omega_z$
centered at $z$ with some radius such that its closure
$\overline\Omega_z$ is contained in $\H$. Let $\Omega_z'$ be the open
disc with half the radius. By compactness of $K$ the cover
$\{\Omega_z'\}$ admits a finite subcover i.e. $K\subset
\cup_{i=1}^n\Omega_{z_i}$  for $z_i\in K$. Choose as a fundamental
domain for $\Gamma\backslash\H$  a normal polygon $F$. Since $\Gamma$
is a discrete subgroup of $\slr$, $\Omega_{z_i}$ intersects non-trivially with $\gamma F$ for
only finitely many (say $n_i$) $\gamma\in \Gamma$ (See \cite[1.6.2
(3)]{Miyake:2006a}).

Therefore, for any automorphic function $f$, 
\begin{align}
  \label{eq:23a}
 \nonumber \norm{f}^2_{L^2(\Omega_{z_i})}:=&\int_{\Omega_{z_i}}\abs{f(z)}^2dxdy\\
  &\leq
  n_i\overline y_i^2\int_{F}\abs{f(z)}^2d\mu(z)= n_i\overline y_i^2\norm{f}_2^2
\end{align}
and 
\begin{align}
  \label{eq:23b}
 \nonumber \norm{\Delta_E
    f}^2_{L^2(\Omega_{z_i})}:=&\int_{\Omega_{z_i}}\abs{\Delta_E
    f(z)}^2dxdy\\ & \leq
  n_i\underline y_i^{-2}\int_{F}\abs{\Delta f(z)}^2d\mu(z)= n_i\underline
  y_i^{-2}\norm{\Delta f}_2^2
\end{align}
where $\overline y_i<\infty$ and $\underline y_i>0$ are heights over and under $\Omega_i$.
It is straightforward to verify that  $G_n(z,s)$ is in
$W^{1,2}(\Omega_i)$ (since it is continuously differentiable) and that $\Omega_i$ has the cone property, so we may use the above
inequalities to conclude
\begin{align*}
\sup_{z\in K}&\abs{G_n(z,t)}\leq  \max_{i=1}^n \sup_{z\in
  \Omega_{z_i}'}\abs{G_n(z,s)}\\
& \leq  \max_i C_i \norm{G_n(z,s)}_{W^{2,2}(\Omega_{z_i}')}
&\textrm{by (\ref{eq:21})} \\
& \leq  \max_i C_iC_i'C_i'' (\norm{G_n(z,s)}_{L^2(\Omega_{z_i})}+\norm{\Delta_E G_n(z,s)}_{L^2(\Omega_{z_i})}) &\textrm{by (\ref{eq:22})}\\
& \leq  \max_i C_iC_i'C_i''(\norm{G_n(z,s)}_2+\norm{\Delta G_n(z,s)}_2) &\textrm{  by (\ref{eq:23a}) and (\ref{eq:23b})}\\
&\leq C_K (\abs{t}(n^2+\abs{t}^2))&\textrm{ by (\ref{eq:for15}) and (\ref{eq:15})}
\end{align*}
which concludes the proof.
\end{proof}
We note that Lemma \ref{growth} implies that
\begin{equation}
    G_n(z,s)=O(\abs{t}^3)
\end{equation}
when $\abs{n}\leq \abs{t}$,
and by applying the Phragm\'{e}n-Lindel\"of theorem we may reduce the exponent to 
$\max(6(1-\sigma)+\varepsilon,0)$ for any $\varepsilon>0$.

We may now use the meromorphic continuation of $G_n(z,s)$ and Lemma \ref{growth} to get an asymptotic expansion with error term for the sum in (\ref{eq:12}). We will assume that there are no exceptional eigenvalues, which implies that $G_n(z,s)$ is regular in $\Re(s)>1/2$. If this is not the case $G_n(z,s)$ will still be regular in $\Re(s)>h$ for some  $h<1$. In (\ref{cauchy}) below we then move the line of integration to $\Re(s)>h+\varepsilon$. Proceeding with the obvious changes still gives a nontrivial error term in the end. We shall not dwell on the details.

Let $\psi_U:\R_+\to\R$, $U\geq U_0$,  be a family of smooth non-increasing functions with \begin{equation}\label{krebsgang}\psi_U(t)=\begin{cases}1 &\textrm{ if }t\leq1-1/U\\0 &\textrm{ if }t\geq1+1/U,  \end{cases}\end{equation} and $\psi_U^{(j)}(t)=O(U^j)$ as $U\to\infty$. For $\Re(s)>0$ we let \begin{equation*}M_U(s)=\int_0^\infty\psi_U(t)t^{s-1}dt\end{equation*}be the Mellin transform of $\psi_U$. Then we have \begin{equation}\label{vud1}M_U(s)=\frac{1}{s}+O\left(\frac{1}{U}\right)\qquad\textrm{as }U\to\infty \end{equation} and for any $c>0$  \begin{equation}\label{vud2}M_U(s)=O\left(\frac{1}{\abs{s}}\left(\frac{U}{1+\abs{s}}\right)^c\right)\qquad\textrm{as }\abs{s}\to\infty.\end{equation} Both estimates are uniform for $\Re(s)$ bounded.   The first is a mean value estimate while the second is successive partial integration and a mean value estimate. We use here the estimate $\psi_U^{(j)}(t)=O(U^j)$. 
The Mellin inversion formula now gives
\begin{align}\sum_{\gamma\in \G}e(n\p(\g z)/\pi)\psi
_U\left(\frac{\,\cosh(r(\g z))
}{R}\right)=\frac{1}{2\pi
 i}\int_{\Re(s)=2}\!\!\!G_n(z,s)M_U(s)R^sds.
\end{align} We note that by Lemma \ref{growth} the integral is convergent as long as $G_n(z,s)$ has polynomial growth on vertical lines. We now move the line of integration to the line $\Re(s)=h$ 
 with $h<1$ by integrating along a box of some height and then letting this height go to infinity. Using Lemma \ref{growth} we find that the contribution from the horizontal sides goes to zero. Assume that $s=1$ is the only pole of the integrand with $\Re(s)\geq 1/2+\varepsilon$. Then using Cauchy's residue theorem we obtain
\begin{align}
\nonumber\frac{1}{2\pi i}&\int_{\Re(s)=2}\!\!\!G_n(z,s)M_U(s)R^sds\\ \label{cauchy}&=\hbox{Res}_{s=1}\left(G_n(z,s)M_U(s)R^s\right)+\frac{1}{2\pi i}\int_{\Re(s)=1/2+\varepsilon}\!\!\!G_n(z,s)M_U(s)R^sds\\
\nonumber&=\delta_{n=0}\left(\frac{2\pi R}{\vol(\GmodH)}+O(R/U)\right)+\frac{1}{2\pi i}\int_{\Re(s)=1/2+\varepsilon}\!\!\!G_n(z,s)M_U(s)R^sds.
\end{align}
If there are other small eigenvalues there are additional main terms. In bypassing we note that their coefficients will depend on the $n$-th hyperbolic Fourier coefficients of the eigenfunctions corresponding to small eigenvalues. (See \cite[Theorem 4 p. 116]{Good:1983b}.)  
If we choose $c=3+\varepsilon$ and use Lemma \ref{growth} the last
integral is $O(R^{1/2+\varepsilon}U^{3+\varepsilon} (n^2+1) )$. The
interval with $\abs{\Im(s)} \leq 1$ can easily be dealt with using the
bound
\begin{align*}
\Vert R(s)\Vert_\infty \le \max_j \left\lvert \frac{1}{\sigma(1-\sigma)}-
\frac{1}{\sigma_j(1-\sigma_j)}\right\rvert,
\end{align*}
which in turn gives us an estimate for $G_n(z,s)$.

 If $n=0$ we see that by further requiring $\psi_U(t)=0$ if $t\geq 1$ and  $\tilde\psi_U(t)=1$ if $t\leq 1$, we have
 \begin{equation*}
   \sum_{\gamma\in \G}\psi
_U\left(\frac{\,\cosh(r(\g z))
}{R}\right)\leq\sum_{\substack{\gamma\in \G\\ \cosh(r(\g z))\le R }} 1\leq \sum_{\gamma\in \G}\tilde\psi
_U\left(\frac{\,\cosh(r(\g z))
}{R}\right).
 \end{equation*}
Choosing $U=R^{1/8}$ we therefore obtain:
\begin{lemma}\label{usuallattice}
With assumptions as above we have
  \begin{align}\label{eq:19}
    \#\{ \g \in \G \vert \cosh (r(\g z))\leq R \}=\frac{2\pi R}{\vol (\GmodH)}+O(R^{7/8+\varepsilon}).
  \end{align}
\end{lemma}
We note that this implies (\ref{hyperbolic lattice}) with $\a=7/8$. Using this we can now deal with the general case. To get from a smooth cut-off to a sharp one we notice that if $\psi_U(t)=1$ for $t\leq 1$ then we may bound the difference
\begin{equation*}
  \sum_{\gamma\in \G}e(n\p(\g z)/\pi)\psi_U\left(\frac{\cosh (r(\g z))
}{R}\right)-\sum_{{\substack{\gamma\in \G\\ \cosh (r(\g z))\leq R}}}e(n\p(\g z)/\pi)=O\biggm( \!\!\!\!\!\sum_{\substack{\gamma\in \G\\ R<\cosh (r(\g z))\leq R(1+1/U)}}\!\!\!\!\! 1\biggm)
\end{equation*}
which by Lemma \ref{usuallattice} is $O(R/U+R^{7/8+\varepsilon})$. Combining the above we find that for $n\neq 0$
\begin{equation*}
 \sum_{\substack{\g\in \G\\\cosh(r(\g z))\leq R}}\!\! e(n\p(\g z)/\pi) =O(R^{1/2+\varepsilon}U^{3+\varepsilon}(n^2+1)+R/U+R^{7/8+\varepsilon}).
\end{equation*}
Using the Erd\"os-Tur\'an inequality \cite[Theorem 3]{ErdosTuran:1948a}  we find that 

\begin{align*}
  \frac{\#\{\g\in \G \vert  \cosh (r(\g z))\leq R,\ \p(\g z)/\pi\in I \}}{ \#\{\g\in \G \vert  \cosh (r(\g z))\leq R\ \}}=\abs{I}\\ +O(1/M+R^{-1/2+\varepsilon}U^{3+\varepsilon}M^2&+\log M (1/U+R^{-1/8+\varepsilon}))
\end{align*}
for any $M$. Letting $M=U=R^{1/12}$ we arrive at the following (still assuming that there are no small eigenvalues):
\begin{theorem} 
For all $\varepsilon > 0$ and $I \subset \R \slash \Z$ we have
  \begin{equation*}
     \frac{\#\{\g\in \G \vert \cosh (r(\g z))\leq R,\, \p(\g z)/\pi\in I \}}{\#\{\g\in \G \vert \cosh (r(\g z))\leq R\ \}}=\abs{I}+O(R^{-1/12+\varepsilon}).
  \end{equation*}
\end{theorem}
Theorem \ref{effective equidistribution} follows easily.
\section{Proof of Theorem \ref{boca-generalized}}

We wish to find the limiting distribution of the number of lattice points in angular sectors defined from $z_0$ when ordering the lattice points $\gamma w$ according to the distance to $z_1$. More precisely we want to find the asymptotics of 
 \begin{align}\label{again}
 \mathscr{N}_\G^I(R,z_0,z_1,w) = \# \{ \g \in \Gamma  \vert d(z_1,\gamma w)\le R, \, \varphi_{{z_0},w}(\g) \in I\}.
 \end{align}

Our strategy for finding the asymptotics is the following:  We find the hyperbolic distance from $z_0$ to the intersection(s) between the hyperbolic circle with center at $z_1$ and radius $R$ and the geodesic through $z_0$ determined by an angle $t\in [-\pi,\pi]$ relative to the vertical geodesic through
$z_0$. Once we have an an asymptotic expression for this distance we can make a Riemann sum approximation of the counting function (\ref{again}). The summands can be estimated via Theorem \ref{effective equidistribution} leading to a proof of Theorem \ref{boca-generalized}.

We may safely assume that $z_0 = i$ -- it is easy to extend our
results to the general case. We would like to find the distance from
$i$ to the relevant intersection point which will be denoted by $w' =
x' + iy'$. There are 2 intersection points, but we choose the one which has negative real part for $t > 0$. This distance will be denoted $Q(z_1,t,R)$. 
\begin{figure}[h]
\begin{center}
\includegraphics{fig3.1}
\caption{}  
\label{computing Q}
\end{center}
\end{figure}

Now fix $z_1$, $t$  and $R$. Let $\alpha \in \R$ and $\delta \in \R_+$ denote the center and the radius respectively of the Euclidean half-circle which is the geodesic through $i$ and $w'$. From Figure \ref{computing Q} it is clear that 
\begin{equation}\delta = 1/\vert \sin(t)\vert, \quad \alpha = -\cot (t)\end{equation} if $t \ne 0,\pm \pi$. Thus we see that
\begin{align}\label{yprimecoord}
y' = \sqrt{\delta^2 -(x'-\alpha)^2} = \sqrt{1-x'^2+2\alpha x'}.
\end{align}
On the other hand it is well-known that the locus of points on the hyperbolic circle with center at $x_1 + iy_1$ and radius $R$ is determined by the equation
\begin{align*}
\vert x_1 + i y_1 \cosh(R) - z\vert = y_1 \sinh(R),
\end{align*}
which is equivalent to
\begin{align*}
x^2 + y^2 + x_1^2 -2xx_1 +y_1^2 = 2y_1y \cosh(R).
\end{align*}
Using the expression for $y'$ given in (\ref{yprimecoord}) we obtain the equation
\begin{equation}\label{quadratic}
\frac{\beta}{2} + (\alpha-x_1)x' =  y_1 \cosh(R)\sqrt{\delta^2 -(x'-\alpha)^2}
\end{equation}
for $x'$, where $\beta = \vert z_1 \vert^2 + 1$. By squaring (\ref{quadratic}) we get the quadratic equation
\begin{align*}
\left(\frac{(\alpha-x_1)^2}{y_1^2 \cosh^2(R)}+1\right)x'^2 + \left(\frac{\beta(\alpha-x_1)}{y_1^2 \cosh^2(R)}-2\alpha\right)x' +\frac{\beta^2}{4 y_1^2 \cosh^2(R)} - 1 = 0,
\end{align*}
with the solution
\begin{align}\label{xprimecoord}
x' = \frac{\alpha-\frac{\beta(\alpha-x_1)}{2y_1^2 \cosh^2(R)}-\fortegn(t)\sqrt{\delta^2 + \frac{(\alpha-x_1)^2}{y_1^2 \cosh^2(R)}-\frac{\beta^2}{4y_1^2 \cosh^2(R)}-\frac{\alpha\beta(\alpha-x_1)}{y_1^2 \cosh^2(R)}}}{1+\left(\frac{\alpha - x_1}{y_1 \cosh(R)}\right)^2}.
\end{align}
Naturally, the quadratic equation has 2 solutions, but the solution
above is the intersection point we are interested in.
The distance $Q(z_1,t,R)$ is
\begin{align}
\label{Q-expr}Q(z_1,t,R) &= \log\left(\frac{\vert w' + i\vert+\vert w'-i \vert}{\vert w'+i\vert-\vert w'-i\vert}\right).
\end{align}
We note that
\begin{align}
\nonumber\frac{\vert w' + i\vert+\vert w'-i \vert}{\vert w'+i\vert-\vert w'-i\vert} &= \frac{x'^2 + y'^2 + 1 +\sqrt{(x'^2+y'^2+1)^2-4y'^2}}{2y'}\\
\label{general formula}
&= \frac{1+\alpha x' + \delta \vert x'\vert}{y'}\\
\nonumber&= \frac{1+\alpha x' - \delta' x'}{y'},
\end{align}
where $\delta' = 1/\sin(t)$. Using Taylor's formula with remainder we see that
\begin{align*}
\fortegn(t)\sqrt{\delta^2 +\frac{(\alpha-x_1)^2}{y_1^2 \cosh^2(R)}-\frac{\beta^2}{4y_1^2 \cosh^2(R)}-\frac{\alpha\beta(\alpha-x_1)}{y_1^2 \cosh^2(R)}}=\\
\delta' +\frac{\frac{(\alpha-x_1)^2}{y_1^2 \cosh^2(R)}-\frac{\beta^2}{4y_1^2 \cosh^2(R)}-\frac{\alpha\beta(\alpha-x_1)}{y_1^2 \cosh^2(R)}}{2\delta'}+O\left(\frac{\delta}{\cosh^4(R)}\right)
\end{align*}
as $R \to \infty$, where the constant implied depends on $z_1$. From this and (\ref{general formula}) we deduce that
\begin{align}
x' = \frac{\alpha - \delta'}{1+\left(\frac{\alpha - x_1}{y_1 \cosh(R)}\right)^2}+ \frac{O((1+\delta)e^{-R})}{1+\left(\frac{\alpha - x_1}{y_1 \cosh(R)}\right)^2}
\end{align}
and hence
\begin{align}
1+\alpha x' - \delta' x' = 1 + \frac{(\alpha - \delta')^2}{1+\left(\frac{\alpha - x_1}{y_1 \cosh(R)}\right)^2}+ \frac{O(\delta(1+\delta)e^{-R})}{1+\left(\frac{\alpha - x_1}{y_1 \cosh(R)}\right)^2}.
\end{align}
This implies that
\begin{align}\label{enum}
\sin^2(t)\left(1+\left(\frac{\alpha - x_1}{y_1 \cosh(R)}\right)^2\right)(1+\alpha x' - \delta' x') = 2+2\cos (t) + O(e^{-R}).
\end{align}

Now we look at
\begin{align*}
y'^2 \left(1+\left(\frac{\alpha - x_1}{y_1
      \cosh(R)}\right)^2\right).
\end{align*}
Using Taylor's formula as before we get
\begin{align*}
y'^2 &\left( 1 + \left(\frac{\alpha - x_1}{y_1 \cosh(R)}\right)^2
\right)^2 \\
={}&\left( 1 + \left(\frac{\alpha - x_1}{y_1 \cosh(R)}\right)^2 \right)^2 - \Biggr(\alpha-\frac{\beta(\alpha-x_1)}{2y_1^2 \cosh^2(R)}-\\ 
&\fortegn(t)\sqrt{\delta^2 +\frac{(\alpha-x_1)^2}{y_1^2 \cosh^2(R)}-\frac{\beta^2}{4y_1^2 \cosh^2(R)}-\frac{\alpha\beta(\alpha-x_1)}{y_1^2 \cosh^2(R)}}\Biggr)^2+\\
& 2\alpha\left( 1 + \left(\frac{\alpha - x_1}{y_1 \cosh(R)}\right)^2 \right)\Biggr(\alpha-\frac{\beta(\alpha-x_1)}{2y_1^2 \cosh^2(R)}-\\
&\fortegn(t)\sqrt{\delta^2 +\frac{(\alpha-x_1)^2}{y_1^2 \cosh^2(R)}-\frac{\beta^2}{4y_1^2 \cosh^2(R)}-\frac{\alpha\beta(\alpha-x_1)}{y_1^2 \cosh^2(R)}}\Biggr)\\
={}&\frac{1}{y_1^2\cosh^2(R)}\biggr(\frac{\beta^2}{4}+(\alpha-x_1)^2+\alpha\beta (\alpha-x_1)+2\alpha^2(\alpha-x_1)^2-\\
&\delta'(\alpha-x_1)(\beta+2\alpha(\alpha-x_1))
\biggm)+O\left(\frac{\delta^4}{\cosh^4(R)}\right)\\
={}& \frac{(\beta - (\beta-2)\cos(t)+2x_1\sin(t))^2(1+\cos(t))^2}{4 y_1^2\cosh^2(R)\sin^4(t)}+O\left(\frac{\delta^4}{\cosh^4(R)}\right)
\end{align*}
as $R \to \infty$. From this we conclude that
\begin{align}
\nonumber\frac{1+\cos(t)}{2y'\left( 1 + \left(\frac{\alpha - x_1}{y_1 \cosh(R)}\right)^2 \right) \sin^2(t)} &=\frac{y_1\cosh(R)}{\beta-(\beta-2)\cos(t)+2x_1\sin(t)}+ O(e^{-4R})\\
\label{another expansion}
&=\frac{y_1 e^R}{2(\beta-(\beta-2)\cos(t)+2x_1\sin(t))} + O(e^{-R}).
\end{align}
We are interested in $e^{Q(z_1,t,R)}$. Combining (\ref{general formula}), (\ref{Q-expr}), (\ref{another expansion}) and (\ref{enum}) we conclude that
\begin{align}\label{R_distance}
e^{Q(z_1,t,R)} = \frac{1+\alpha x' + \delta' x'}{y'} = \frac{2y_1 e^R}{\beta-(\beta-2)\cos(t)+2x_1\sin(t)} + O(1).
\end{align}

To finish the proof we use the following elementary lemma which \lq
integrates\rq{} 
Theorem \ref{effective equidistribution} over more general regions:
\begin{lemma}\label{Riemann sums} Let $D(R,\theta): \R_+ \times \R/\Z \to \R_+$ be a
  function which satisfies $e^{D(R,\theta)}=k(\theta)e^R+O(e^{\beta R})$
  for some $\beta<1$ uniformly in $\theta$. Assume that $k(\theta)\in
  C^1(\R/\Z)$. Then as $R\to\infty$
\begin{align*}N_{\Gamma,D}^I(R,z_0,z_1) :&= \#\{ \gamma \in \G \mid d(z_0, \gamma z_1) \le D(R,\varphi_{z_0,z_1}(\gamma)), \ \varphi_{z_0,z_1}(\gamma) \in I \} \\
&=  \frac{\kappa_\G\pi}{ \vol(\GmodH)} \int_I k(\theta) d\theta e^R+O(e^{\delta R})
 \end{align*}
for some $\delta<1$.
\end{lemma}
\begin{proof} Let $B=B(R)$ be a integer valued function of $R$ to be determined
  later. For each integer $j\leq B$ we choose 
$\omega_j,
\omega^j \in \left[a+\frac{(j-1)(b-a)}{B},a+\frac{j(b-a)}{B}\right]$
such that  
\begin{align*}
k(\omega_j) = \inf \left\{ k(\omega) \biggm\vert
  \omega \in \left[a+\frac{(j-1)(b-a)}{B},a+\frac{j(b-a)}{B}\right] \right\}
\end{align*}
and
\begin{align*}
k(\omega^j) = \sup \left\{ k(\omega) \biggm\vert
  \omega \in \left[a+\frac{(j-1)(b-a)}{B},a+\frac{j(b-a)}{B}\right] \right\}.
\end{align*}
We split the interval in $B$ equal intervals (and compensate for
counting the
endpoints twice) to get
\begin{align*}N_{\Gamma,D}^I(R,z_0,z_1)=\sum_{j=0}^B&N_{\Gamma,D}^{[a+\frac{(j-1)}{B}(b-a),a+\frac{j}{B}(b-a)]}(R,z_0,z_1)\\
  &-\sum_{j=1}^{B-1}N_{\Gamma,D}^{[a+\frac{j}{B}(b-a),a+\frac{j}{B}(b-a)]}(R,z_0,z_1).
\end{align*}
 The last sum is $O(Be^{\alpha R})$ by Theorem \ref{effective
   equidistribution} and the assumption on $D(R,\theta)$. The first sum can
 be evaluated as follows. By using Theorem \ref{effective
   equidistribution} again we have 
\begin{align*}
    \frac{\kappa_\G\pi(b-a)}{B \vol(\GmodH)} \omega_je^R-Ce^{\alpha
      R}\leq &\\ N_{\Gamma,D}^{[a+\frac{(j-1)}{B}(b-a),a+\frac{j}{B}(b-a)]}&(R,z_0,z_1)\\&\leq
   \frac{\kappa_\G\pi(b-a)e^R}{B \vol(\GmodH)}\omega^j+Ce^{\alpha R}.
 \end{align*}
Summing this inequality we find the Riemann sums
$$\sum_{j=0}^B\omega_j\frac{(b-a)}{B}, \quad \sum_{j=0}^B\omega^j\frac{(b-a)}{B}.$$
Since $k$ is $C^1$ these converge to $\int_Ik(\theta)d\theta$ with rate $O(1/B)$ as is
seen using the mean value theorem. We therefore find  that 
$$N_{\Gamma,D}^I(R,z_0,z_1)= \frac{\kappa_\G\pi}{ \vol(\GmodH)}\int_I
k(\theta) d\theta e^R+O(e^R/B)+O(Be^{\alpha R}).$$
 Balancing the error terms we get the result.

\end{proof}

We can now finish the proof of Theorem \ref{boca-generalized}. Let $\rho_{z_0,z_1}(\omega)$ denote the fraction
\begin{align*}
\frac{2y_0 y_1}{((x_0-x_1)^2+y_0^2+y_1^2)(1-\cos(2\pi\omega))+2y_0^2\cos(2\pi\omega)+2(x_1-x_0)y_0\sin(2\pi\omega)}.
\end{align*}

We start with the case $z_0 = i$. Equation (\ref{R_distance}) allows
us to use Lemma \ref{Riemann sums} which gives Theorem
\ref{boca-generalized} immediately. The general case
can easily be reduced to the case where $z_0 = i$ by conjugation of
$\Gamma$ with the element $\left(\begin{smallmatrix} \sqrt{y_0} &
    x_0/\sqrt{y_0}\\ 0 & 1/\sqrt{y_0} \end{smallmatrix}\right)$. This
finishes the proof of Theorem 3.\qed\\

\noindent{\bf Acknowledgements:} We thank the anonymous referee for
his/her useful comments. 

\bibliographystyle{plain}
\bibliography{minbibliography.bib}

\def\cprime{$'$}
\begin{thebibliography}{10}

\bibitem{AdamsFournier:2003a}
Robert~A. Adams and John J.~F. Fournier.
\newblock {\em Sobolev spaces}, volume 140 of {\em Pure and Applied Mathematics
  (Amsterdam)}.
\newblock Elsevier/Academic Press, Amsterdam, second edition, 2003.

\bibitem{Boca:2007a}
Florin~P. Boca.
\newblock Distribution of angles between geodesic rays associated with
  hyperbolic lattice points.
\newblock {\em Q. J. Math.}, 58(3):281--295, 2007.

\bibitem{Delsarte:1942a}
Jean Delsarte.
\newblock Sur le gitter fuchsien.
\newblock {\em C. R. Acad. Sci. Paris}, 214:147--179, 1942.

\bibitem{ElstrodtGrunewaldMennicke:1988a}
J{\"u}rgen Elstrodt, Fritz Grunewald, and Jenspa Mennicke.
\newblock {Arithmetic Applications of the Hyperbolic Lattice Point Theorem}.
\newblock {\em Proc. London Math. Soc.}, s3-57(2):239--283, 1988.

\bibitem{ErdosTuran:1948a}
Paul Erd{\"o}s and Paul Tur{\'a}n.
\newblock On a problem in the theory of uniform distribution. {I}.
\newblock {\em Nederl. Akad. Wetensch., Proc.}, 51:1146--1154, 1948 = Indagationes
  Math. 10, 370--378, 1948.

\bibitem{Good:1983b}
Anton Good.
\newblock {\em Local analysis of {S}elberg's trace formula}, volume 1040 of
  {\em Lecture Notes in Mathematics}.
\newblock Springer-Verlag, Berlin, 1983.

\bibitem{Gunther:1980a}
Paul G\"unther.
\newblock Gitterpunktprobleme in symmetrischen {R}iemannschen {R}\"aumen vom
  rang 1.
\newblock {\em Math. Nachr.}, 94:5--27, 1980.

\bibitem{Huber:1959a}
Heinz Huber.
\newblock Zur analytischen {T}heorie hyperbolischen {R}aumformen und
  {B}ewegungsgruppen.
\newblock {\em Math. Ann.}, 138:1--26, 1959.

\bibitem{Huber:1960a}
Heinz Huber.
\newblock Zur analytischen {T}heorie hyperbolischer {R}aumformen und
  {B}ewegungsgruppen. {II}.
\newblock {\em Math. Ann.}, 142:385--398, 1960/1961.

\bibitem{Huber:1961a}
Heinz Huber.
\newblock Zur analytischen {T}heorie hyperbolischer {R}aumformen und
  {B}ewegungsgruppen. {II} nachtrag.
\newblock {\em Math. Ann.}, 143:463---464, 1961.

\bibitem{Iwaniec:2002a}
Henryk Iwaniec.
\newblock {\em Spectral methods of automorphic forms}, volume~53 of {\em
  Graduate Studies in Mathematics}.
\newblock American Mathematical Society, Providence, RI, second edition, 2002.

\bibitem{Jost:2002a}
J{\"u}rgen Jost.
\newblock {\em Partial differential equations}, volume 214 of {\em Graduate
  Texts in Mathematics}.
\newblock Springer-Verlag, New York, 2002.
\newblock Translated and revised from the 1998 German original by the author.

\bibitem{Kato:1976a}
Tosio Kato.
\newblock {\em Perturbation theory for linear operators}.
\newblock Springer-Verlag, Berlin, second edition, 1976.
\newblock Grundlehren der Mathematischen Wissenschaften, Band 132.

\bibitem{LaxPhillips:1982a}
Peter~D. Lax and Ralph~S. Phillips.
\newblock The asymptotic distribution of lattice points in {E}uclidean and
  non-{E}uclidean spaces.
\newblock {\em J. Funct. Anal.}, 46(3):280--350, 1982.

\bibitem{Levitan:1987a}
Boris~M. Levitan.
\newblock {Asymptotic formulae for the number of lattice points in Euclidean
  and Lobachevskij spaces.}
\newblock {\em Russ. Math. Surv.}, 42(3):13--42, 1987.

\bibitem{Margulis:1969a}
Grigory~A. Margulis.
\newblock Certain applications of ergodic theory to the investigation of
  manifolds of negative curvature.
\newblock {\em Funkcional. Anal. i Prilo\v zen.}, 3(4):89--90. English
  translation in Functional Anal. Appl. 3 (1969), no. 4, 335--336, 1969.

\bibitem{Miyake:2006a}
Toshitsune Miyake.
\newblock {\em Modular forms}.
\newblock Springer Monographs in Mathematics. Springer-Verlag, Berlin, english
  edition, 2006.
\newblock Translated from the 1976 Japanese original by Yoshitaka Maeda.

\bibitem{Murty:2001a}
M.~Ram Murty.
\newblock {\em Problems in analytic number theory}, volume 206 of {\em Graduate
  Texts in Mathematics}.
\newblock Springer-Verlag, New York, 2001.
\newblock Readings in Mathematics.

\bibitem{Nicholls:1983a}
Peter Nicholls.
\newblock A lattice point problem in hyperbolic space.
\newblock {\em Michigan Math. J.}, 30(3):273--287, 1983.

\bibitem{Patterson:1975a}
Samuel~J. Patterson.
\newblock A lattice-point problem in hyperbolic space.
\newblock {\em Mathematika}, 22(1):81--88, 1975.

\bibitem{PhillipsRudnick:1994a}
Ralph Phillips and Ze{\'e}v Rudnick.
\newblock The circle problem in the hyperbolic plane.
\newblock {\em J. Funct. Anal.}, 121(1):78--116, 1994.

\bibitem{Selberg:1989a}
Atle Selberg.
\newblock {\em G\"ottingen lecture notes in Collected papers. {V}ol. {I}}.
\newblock Springer-Verlag, Berlin, 1989.
\newblock With a foreword by K. Chandrasekharan.

\end{thebibliography}

\end{document}